\documentclass{amsart} % Use the aims.cls file to compile your paper
\usepackage{amsmath}
\usepackage{paralist}
\usepackage[misc]{ifsym}
\usepackage{epsfig} %For pictures: screened artwork should be set up with an 85 or 100 line screen
\usepackage{epstopdf} %This is to transfer .eps figure to .pdf figure; please compile your article using PDFLaTex or PDFTeXify.
\usepackage[colorlinks=true]{hyperref}
\hypersetup{urlcolor=blue, citecolor=red}
\allowdisplaybreaks

\newtheorem{theorem}{Theorem}[section]
\newtheorem{corollary}[theorem]{Corollary}

\newtheorem{lemma}[theorem]{Lemma}

\theoremstyle{definition}

\newtheorem{remark}[theorem]{Remark}

\def\dist{\mathop{\text{\normalfont dist}}}

\title{Hopf's lemmas and boundary point results  for the fractional $p$-Laplacian}
\author[Pablo Ochoa and Ariel Salort]{}


\thanks{The first author was partially supported by SIIP-UNCUYO  B017-TI. The second author was partially supported by ANPCyT under grant PICT 2019-3530. Both authors are members of CONICET.}

\thanks{$^*$Corresponding author: Ariel Salort}

\begin{document}
\maketitle

% Enter the first author's name and email address; email addresses are required for each author.
% Use footnote notations to indicate address and affiliations with commas between numbers if more than one address applies; see below for examples.
\centerline{\scshape
Pablo Ochoa $^{{\href{mailto:pablo.ochoa@ingenieria.uncuyo.edu.ar}{\textrm{\Letter}}}1}$
and Ariel Salort $^{{\href{mailto:asalort@dm.uba.ar}{\textrm{\Letter}}}*2}$}

\medskip

{\footnotesize
% Enter the full affiliation and country name:
% Do not insert commas or periods at the end of lines.
 \centerline{$^1$Universidad Nacional de Cuyo, Fac. de Ingenier\'ia. CONICET. Universidad J. A. Maza, Parque Gral. San Mart\'in 5500\\
 %Mendoza, Argentina.}
} % Do not forget to end {\footnotesize with the sign }

\medskip

{\footnotesize
 % Enter the full affiliation and country name:
 \centerline{$^2$Instituto de C\'alculo, CONICET, \\
 Departamento de Matem\'atica, FCEN - Universidad de Buenos Aires\\ 
 %Ciudad Universitaria, 0+$\infty$ building, C1428EGA, Av. Cantilo s/n\\
 %Buenos Aires, 
 , Argentina}
}

\bigskip

% The name of the handling editor will be entered by AIMS production staff.
% "Communicated by Handling Editor" is not needed for special issue.
 \centerline{(Communicated by Handling Editor)}

%%%%%%%%%%%%%%%%%%%%%%%%%%%%%%%%%%%%%%%%%%%%%%%%%%%%%%%
%             5. ABSTRACT
%%%%%%%%%%%%%%%%%%%%%%%%%%%%%%%%%%%%%%%%%%%%%%%%%%%%%%%

\begin{abstract}
In this paper, we consider different versions of the classical Hopf's boundary lemma in the setting of the fractional $p-$Laplacian for $p \geq 2$. We start by providing for a new proof  to a  Hopf's lemma   based on comparison principles. Afterwards, we give a Hopf's result for sign-changing potential describing the behavior of the fractional normal derivative of solutions around boundary points. The main contribution here is that we do not need to impose a global condition on the sign of the solution. Applications of the main results to boundary point lemmas and a discussion of  non-local non-linear overdetermined problems are also provided.
\end{abstract}

\section{Introduction}
Hopf's classical boundary lemma offers a refined analysis of the outer normal derivative of superharmonic functions at a minimum point on the boundary of a domain that satisfies the interior ball condition, which is useful for proving a strong minimum principle for second order uniformly elliptic operators. More precisely,  if $u\in C^2(\overline\Omega)$, being $\Omega\subset \mathbb R^n$ open and bounded with the interior ball condition, and $x_0\in \partial\Omega$ is such that $u(x_0)<u(x)$ for all $x\in\Omega$, then
$$
-\Delta u\geq c(x) u \text{ in } \Omega \implies \frac{\partial u}{\partial \nu}(x_0)<0,
$$
where $c\in L^\infty(\Omega)$ is such that $c(x)\leq 0$, and $\partial u / \partial\nu$ is the outer normal derivative of $u$ at $x_0$. More generally, whether or not the normal derivative exists, it holds that
\begin{equation} \label{eq.intro}
\liminf_{\Omega \ni x\to x_0} \frac{u(x)-u(x_0)}{|x-y_0|}>0,
\end{equation}
where the angle between $x_0-x$ and the normal at $x_0$ is less than $\frac{\pi}{2}-\beta$ for some $\beta>0$. 
%See for instance Lemma 3.4 in \cite{GT}.

A nonlocal (and possibly) nonlinear generalization of this result was introduced in \cite{FJ} and \cite{GS} for the well-known fractional Laplacian operator $(-\Delta)^s$, $s\in (0,1)$, which, up to a normalization constant, is defined as
$$
(-\Delta)^s u(x):=2  \text{p.v.} \int_{\mathbb R^n} \frac{u(x)-u(y)}{|x-y|^{n+2s}} \,dy
$$ 
 The authors in \cite{FJ} proved a Hopf's lemma for the entire antisymmetric weak solution to the problem
\begin{equation} \label{eq.lap}
(-\Delta)^s u\geq c(x)u \text{ in } \Omega
\end{equation}
with  $u\geq 0$ in $\mathbb R^n\setminus \Omega$, where $c$ is a $L^\infty(\Omega)$ function. In \cite{GS} it was studied a Hopf's Lemma  for  continuous solutions to \eqref{eq.lap}  under the assumption that $c \in L^\infty(\Omega)$, being $\Omega\subset \mathbb R^n$ an open set satisfying the interior ball condition at $x_0\in\partial\Omega$. Here, the main difference with the local case lies in the fact that the normal derivative of $u$ at a point $x_0\in\partial\Omega$ is replaced with the limit of the ratio $u(x)/\delta_R(x)^s$, where $\delta_R$ is the distance from $x$ to $\partial B_R$, being $B_R$ an interior ball at $x_0$. More precisely, in \cite{GS} it is proved that  under the mentioned assumptions of $\Omega$, $u$ and $c$: 
\begin{itemize}
\item[(i)] if $\Omega$ is bounded, $c \leq 0$ in $\Omega$  and $u\geq 0$ in $\mathbb R^n\setminus \Omega$, then either $u$ vanishes identically in $\Omega$ or
$$
\liminf_{B_R\ni x\to x_0}\frac{u(x)}{\delta_R^s(x)}>0;
$$
\item[(ii)] if $u\geq 0$ in $\mathbb R^n$, then either $u$ vanishes identically in $\Omega$, or the expression above holds true.
\end{itemize}

Later, these results were generalized for the nonlinear counterpart of \eqref{eq.lap} given in terms of the fractional $p-$Laplacian operator, which for $p\in (1,\infty)$ and $s\in (0,1)$ is  defined, up to a normalization constant,  as
$$
(-\Delta_p)^s u(x):=2  \text{p.v.} \int_{\mathbb R^n} \frac{g_p(u(x)-u(y))}{|x-y|^{n+sp}} \,dy
$$ 
being $g_p(t)=|t|^{p-2}t$, $t\in \mathbb R$ . More precisely, in \cite{DPQ} for any $p\in (1,\infty)$,  assuming that  $\Omega$ fulfills the interior ball condition,  for any  continuous weak solution $u$ to 
\begin{equation} \label{eq.plap}
(-\Delta_p)^s u\geq c(x)g_p(u) \text{ in } \Omega, 
\end{equation}
such that $u\geq 0$ in $\mathbb R^n\setminus \Omega$, being $c$ a continuous functions in $\overline \Omega$, then conclusions $(i)$ and $(ii)$ still true.

Similar results for the so-called regional fractional Laplacian were recently proved in \cite{AFT}. Moreover, different versions of the Hopf's Lemma for anti-symmetric functions on a half space were established in \cite{JL}   and \cite{LC}. The case when the right hand side in \eqref{eq.plap} is 0 was treated in \cite{CLQ}.

\medskip

Suppose now that $u$ is a sign-changing solution of a local elliptic problem in a domain, and assume that $u$ does not change sign in a neighborhood of a boundary point $x_0$ with $u(x_0)=0$. Since the Hopf's lemma works under local assumptions, it can be claimed that  $\frac{\partial u}{\partial \nu}(x_0) \neq 0$ unless $u\equiv 0$ 
%(see Lemma 3.4 in \cite{GT}). 
In the recent paper \cite{DSV}, the authors prove the nonlocal version of this assertion for continuous weak solutions to \eqref{eq.lap} under suitable assumptions on $\Omega$, $u$ and $c$ (see Theorem 1.2 in \cite{DSV}).  The analysis is more subtle than in the linear case, but under additional second order fractional growth assumptions, a similar result for the fractional normal derivative can be obtained.

\medskip

We now discuss the contributions of our paper.  We first give a new proof of the Hopf's lemma for \eqref{eq.plap}.  In our arguments, in contrast with \cite{DPQ}, logarithmic estimates of the solution are not needed. We provide for a self-contained proof which  only uses the fact that the fractional $p-$Laplacian of the distance function  is bounded near the boundary, together with the construction of an appropriate subsolution suggested in \cite{CLQ}.

%After a thorough examination of the proof of the Hopf's lemma presented in \cite{S}, it appears that the subsolution considered may not be effective, raising concerns about the validity of the conclusion of the main result.
%We will subsequently present an alternative proof for the same result to address this issue and ensure the accuracy of the findings. However,  we need to assume a further condition on the $L^\infty$ norm of $c$.

More precisely, in Theorem \ref{main1} we prove that when $c\in C(\overline \Omega)$ is  such that $c(x)\leq 0$ and
$u$ is any weak solution to
\begin{align*} 
\begin{cases}
(-\Delta_p)^s u\geq c(x)g_p(u) &\text{ in }\Omega\\ u > 0 & \text{ in }\Omega\\
u=0 &\text{ in } \mathbb R^n\setminus \Omega,
\end{cases}
\end{align*}
then it holds that
$$
\liminf_{B_R\ni x\to x_0}\frac{u(x)}{\delta_R^s(x)}>0
$$
where $x_0\in\partial\Omega$, being  $\Omega\subset \mathbb R^n$ a bounded  set satisfying the interior ball condition at $x_0$,

\normalcolor

 On the other hand, inspired in \cite{DSV}, we study the behavior of a sign-changing solution of the following nonlocal problem
$$
(-\Delta_p)^s u \geq c(x) g_p(u) \text{ in }\Omega,
$$
being $\Omega\subset \mathbb R^n$ an open subset   satisfying the interior ball condition at $x_0\in\partial\Omega$, and $c\in L^1_{loc}(\Omega)$ is such that $c^- \in L^\infty(\Omega)$   (the function $c$ could change its sign).   If there exists $R>0$ such that $u\geq 0$ in $B_R(x_0)$, $u>0$ in $B_R(x_0)\cap \Omega$, then our main result stated in Theorem \ref{Hopf lemma normal derivative} establishes  that for every $\beta\in (0,\frac{\pi}{2})$ it holds that
\begin{equation*}
\liminf_{\Omega\ni x\to x_0}\dfrac{u(x)-u(x_0)}{|x-x_0|^{s}} > 0,
\end{equation*}
whenever the angle between $x-x_0$ and the vector joining $x_0$ and the center of the interior ball is smaller than $\pi/2-\beta$.  Moreover, as applications of Theorem \ref{Hopf lemma normal derivative} we also provide two versions of the classical boundary point lemma (see for instance Section 2.7 in \cite{PS}) in the setting of the fractional $p-$Laplacian. See Theorem \ref{bpl1} and Theorem \ref{bpl2}.

 \medskip

We also aboard a problem where some redundant condition is imposed on the free boundary, which is known as an \emph{overdetermined problem}. In the classical case, if $\Omega\subset \mathbb R^n$ is a bounded domain whose boundary is a priori unknown, %in \cite{S71,W71} 
Serrin and Weinberger proved that if $u$ is the solution of the torsion problem
\begin{align*}
\begin{cases}
-\Delta u=1 &\text{ in } \Omega\\
u=0 &\text{ in }\partial\Omega,
\end{cases}
\end{align*}
with the additional condition  $-\frac{\partial u}{\partial \nu}=\kappa$  along $\partial\Omega$ (there $\kappa$ is a constant and $\nu$ is the outer normal to $\partial\Omega$), then $\Omega$ must be a ball. A related result for the fractional Laplacian can be found in \cite{GS}. However, this situation for the fractional $p-$Laplacian is more subtle. In Theorem \ref{over} we prove that if $u$ is a weak solution to
\begin{align*}
\begin{cases}
(-\Delta_p)^s u=1 & \text{ in } \Omega\\
u=0 &\text{ in } \mathbb R^n\setminus \Omega\\
\displaystyle\lim_{\Omega\ni x\to x_0} \frac{u(x)}{\delta_\Omega^s (x)}=q(|z|) &\text{ for every }x_0\in \partial\Omega,
\end{cases}
\end{align*}
where $\Omega\subset \mathbb R^n$ is open and bounded and $q(r)$ is a non-negative function for $r>0$, and $q$ satisfies a suitable growth behavior, then $\Omega$ must be a ball centered at the origin.

\medskip

 The paper is organized as follows. In Section \ref{prelim}, we give a basic introduction to $p-$fractional Sobolev spaces. We also introduce the notation and preliminary results that will be used throughout the paper. In Section \ref{Hopf principle}, we state and give an alternative proof of the Hopf's principle for the fractional $p-$Laplacian. Next, in Section \ref{hopf lemma sign changing}, we provide the Hopf's lemma for sign-changing potentials. Some consequences of the result are also given. Finally, in Sections \ref{bp theorem} and \ref{overdetermined}, we give applications of the main results to boundary point lemmas and we discuss overdetermined problems for the fractional $p-$Laplacian.

\section{Preliminaries}\label{prelim}
\subsection{Fractional Sobolev spaces}
Let $1<p<\infty$. We define the monotone function $g_p\colon \mathbb R\to \mathbb R$ by
$$
g_p(t)=|t|^{p-2}t.
$$
For  $s\in (0,1)$, $p\in (1,\infty)$ and $\Omega\subseteq\mathbb R^n$, the fractional Sobolev spaces are  defined as 
$$W^{s,p}(\Omega):=\left\{u\colon \mathbb R^n\to \mathbb R \text{ measurable s.t. }  [u]_{W^{s,p}(\Omega)}+ \|u\|_{L^p(\Omega)} <\infty\right\}
$$ 
endowed with the norm
$$
\|u\|_{W^{s,p}(\Omega)}=\|u\|_{L^p(\Omega)}+[u]_{W^{s,p}(\Omega)},
$$
being 
$$
[u]_{W^{s,p}(\Omega)}=\left(\iint_{\Omega\times \Omega} \frac{ |u(x)-u(y)|^p }{|x-y|^{n+sp}}\,dxdy\right)^\frac1p, \qquad \|u\|_{L^p(\Omega)}= \left( \int_\Omega |u|^p\,dx \right)^\frac1p.
$$
Then $W^{s,p}(\Omega)$ with the norm $\|\cdot\|_{W^{s,p}(\Omega)}$ is a reflexive Banach space. In order to consider boundary conditions we also define the space
$$
W^{s,p}_0(\Omega)=\overline{C_c^\infty(\Omega})\subset W^{s,p}(\mathbb R^n),
$$ 
where the closure is taken with respect to the norm $\|\cdot\|_{W^{s,p}(\Omega)}$. When the set $\Omega$ has Lipschitz boundary, the space $W^{s,p}_0(\Omega)$ can be characterized as
$$
W^{s,p}_0(\Omega):=\{u\in W^{s,p}(\mathbb R^n)\colon u=0\text{ a.e. in } \mathbb R^n\setminus \Omega\}.
$$
When $\Omega\subset \mathbb R^n$ is bounded, we also consider the space
\begin{align*}
\widetilde{W}^{s,p}(\Omega)&=\{ u\in L_{sp}(\mathbb R^n)\colon \exists U\colon \Omega\subset\subset U \text{ and } \|u\|_{W^{s,p}(U)}<\infty   \},
\end{align*}
where  the tail $L_{sp}(\mathbb R^n)$ space is defined as
$$
L_{sp}(\mathbb R^n)=\left\lbrace u\in L^1_{loc}(\mathbb R^n)\colon \int_{\mathbb R^n} \frac{ |u(x)|^{p-1} }{(1+|x|)^{n+sp}} \,dx<\infty \right\rbrace.
$$

\medskip

The fractional $p-$Laplacian is defined for any sufficiently smooth function $u:\mathbb{R}^{n}\to \mathbb{R}$ as
$$
(-\Delta_p)^su(x)=2\text{p.v.}\int_{\mathbb{R}^{n}} \frac{|u(x)-u(y)|^p}{|x-y|^{n+sp}}\,dxdy, \quad x \in \mathbb{R}^{n}.
$$
Moreover,  the following representation formula holds
$$
\langle (-\Delta_p)^s u,\varphi \rangle :=  \iint_{\mathbb R^n\times\mathbb R^n} \frac{g_p(u(x)-u(y)) (\varphi(x)-\varphi(y))}{|x-y|^{n+sp}}\,dxdy
$$
for any $\varphi\in W^{s,p}(\mathbb R^n)$.

Finally, we introduce the notion of weak solutions. We say that $u\in \widetilde{W}^{s,p}(\Omega)$ is a \emph{weak super-solution} to $(-\Delta_p)^s u=f$ in $\Omega$ if
$$
\langle (-\Delta_p)^s u,\varphi \rangle \geq \int_\Omega f\varphi\,dx
$$
for any $\varphi\in W^{s,p}_0(\Omega)$ satisfying $\varphi\geq 0$ a.e. in $\Omega$.

Similarly, reversing the inequalities we can define \emph{weak sub-solutions} to  $(-\Delta_p)^s u=f$.

\subsection{The interior ball condition}

Given an open set $\Omega$ of $\mathbb{R}^{n}$ and $x_0 \in \partial \Omega$, we say that $\Omega$ satisfies the  \emph{interior ball condition} at $x_0$ if there is $r_0 >0$ such that, for every $r \in (0, r_0]$, there exists a ball $B_r(x_r) \subset \Omega$ with $x_0 \in \partial \Omega \cap \partial B_r(x_r)$.

\normalcolor

It is well-known that $\Omega$ satisfies the interior ball condition if and only if $\partial\Omega\in C^{1,1}$. 
%{See \cite{LY} for details.

\section{Hopf's principle}\label{Hopf principle}
% Pedimos
% $p^->\frac{1}{1-s}$.
 In this section we provide for  a proof of the Hopf's lemma for supersolutions to $(-\Delta_p)^s u= c(x)g_p(u)$ in $\Omega$, being $c$ a nonpositive function in a bounded set $\Omega$. We also recall that
 $$g_p(t):=|t|^{p-2}t, \quad p \geq 2.$$
 
 The next results generalizes \cite[Theorem 1.3]{CLQ}. 
 
\begin{theorem} \label{main1}
Let $\Omega\subset \mathbb R^n$ be a bounded  set satisfying the interior ball condition at $x_0\in\partial\Omega$, let $c\in C(\overline \Omega)$ be such that $c(x)\leq 0$ in $\Omega$ and let $u\in \widetilde W^{s,p}(\Omega)\cap C(\overline{\Omega})$ be a  weak solution to
\begin{align} \label{eq11}
\begin{cases}
(-\Delta_p) u\geq c(x)g_p(u) &\text{ in }\Omega\\
u>0 &\text{ in } \Omega\\
u\geq 0 &\text{ in } \mathbb R^n\setminus \Omega.
\end{cases}
\end{align}Then
$$
\liminf_{B_R\ni x\to x_0}\frac{u(x)}{\delta_R^s(x)}>0
$$
where $B_R\subseteq\Omega$ and $x_0\in\partial B_R$ and $\delta_R(x)$ is the distance from $x$ to $B_R^c$.
\end{theorem}

\begin{proof}
For a given $x_0\in \partial \Omega$, by the regularity of $\Omega$, %(see for instance \cite{AKSZ}) 
there exists $x_1\in \Omega$ on the normal line to $\partial \Omega$ at $x_0$ and $r_0> 0$ such that
$$
B_{r_0}(x_1)\subset \Omega,\quad  \overline{B}_{r_0}(x_1)\cap \partial \Omega= \{x_0\} \quad \text{and }\quad \text{dist} (x_1,\Omega^{c})=|x_1-x_0|.
$$
We will assume without loss of generality that $x_0= 0$, $r_0=1$ and $x_1= e_n$, with $e_n=(0, \ldots, 0, 1)\in \mathbb{R}^{n}$ and consider a nontrivial weak supersolution $u\in \widetilde{W}^{s,p}(\Omega)\cap C(\overline \Omega)$  to \eqref{eq11}.  

\medskip

Under these assumptions, in  \cite[Theorem 4]{CLQ} it is proved that there exist $r\in (0,1/3\sqrt{5})$ (to be fixed later)  and $C_1>0$ such that 
$$
(-\Delta_p)^s d^s(x)\leq C_1 \quad \text{weakly in } B_1(e_n)\cap B_r(0),
$$
where we have defined the distance function
 $d\colon \mathbb R^n\to \mathbb R$ as $d(x)=\dist (x,B_1^c(e_n))$.

\medskip

We build now a suitable supersolution. Let $D\subset\subset B_1^c(e_n)\cap \Omega$ be a smooth domain and let $\beta>0$ to be determinate. Define 
$$
\underline u(x)= \beta d^s(x)+ \chi_D(x)  u(x).
$$

By \cite[Lemma 2.8]{IMS}, we get

$$(-\Delta_p)^s \underline u \leq \beta^{p-1}C_1 + h \quad \text{weakly in }B_1(e_n)\cap B_r(0),$$where the function $h$ is given by
\begin{equation}
h(x)= 2\int_D \left[g_p\left(\dfrac{\beta d^s(x)-\beta d^s(y)-\chi_D(y)u(y)}{|x-y|^s} \right)-g_p\left(\dfrac{\beta d^s(x)-\beta d^s(y)}{|x-y|^s} \right)\right] \dfrac{dy}{|x-y|^{s+n}}.
\end{equation}Using the inequality
$$g_p(b)-g_p(a)\leq cg_p(b-a), \quad b\leq a,$$
%(see for instance \cite[Lemma 2.1]{MSV21}), 
we obtain
$$h(x)\leq 2c \int_D g_p\left(\dfrac{-u(y)}{|x-y|^s} \right) \dfrac{dy}{|x-y|^{s+n}}.$$

% For any $x\in B_1(e_n)\cap B_r(0)$, it holds that
%\begin{align*}
%(-\Delta_p)^s &\underline u (x)=\lim_{\varepsilon \to 0} \int_{B_\varepsilon^c(x)} \frac{ g_p(\underline u(x)-\underline u(y))}{|x-y|^{n+sp}} \,dy\\
%&=\lim_{\varepsilon \to 0} \int_{B_\varepsilon^c(x)} \frac{g_p(\beta d^s(x)- \beta d^s(y)- u(y) \chi_D(y)) }{|x-y|^{n+sp}}\,dy \\
%&\leq \lim_{\varepsilon \to 0} \int_{B_\varepsilon^c(x)\cap B_1(e_n)} \frac{g_p(   \beta d^s(x)-   \beta d^s(y))}{|x-y|^{n+sp}} \,dy+ 
%  \int_{B_1^c(e_n)} \frac{ g_p( \beta d^s(x)- u(y) \chi_D(y)) }{|x-y|^{n+sp}} \,dy\\
%&=\lim_{\varepsilon \to 0} \int_{B_\varepsilon^c(x)} \frac{ g_p( \beta d^s(x)- \beta d^s(y))}{|x-y|^{n+sp}} \,dy\\ 
%&\quad + 
% \int_{B_1^c(e_n)} \left[\frac{g_p(  \beta d^s(x)- u(y) \chi_D(y)) }{|x-y|^{n+sp}} - \frac{ g_p(\beta d^s(x) -  \beta d^s(y))}{|x-y|^{n+sp}}    \right]\,dy :=(i)+(ii).
%\end{align*}
%The term $(i)$ can be bounded using   Step 2 as
%$$
%(i)= 
%\beta^{p-1}(-\Delta_p)^s d^s(x)  \leq  \beta^{p-1}  C_1.
%$$
Then, since $D\subset\subset B_1^c(e_n)\cap \Omega$, there is $C_D>0$ such that $|x-y|\leq C_D$ for any $x\in B_1(e_n)$ and $y\in D$. Using that $g_p$ is increasing and odd we get
$$h(x)\leq -2cg_r(M_0)\bar{C}_D|D|:= -\widetilde{M}_0,$$
%\begin{align*}
%(ii)&\leq 
%\lim_{\varepsilon \to 0} \int_D \left[\frac{ g_p( \beta d^s(x)-u(y))}{|x-y|^{n+sp}}  - \frac{g_p( \beta d^s(x))}{|x-y|^{n+sp}}    \right]\,dy 
% \leq 
%   \bar C_D \int_D \left( g_p( \beta d^s(x)-u(y) ) - g_p(  \beta d^s(x) )    \right)\,dy,
%\end{align*}
where $\bar C_D:= C_D^{-(n+sp)}$ and  $M_0:=\min_{x\in D} u(x)>0$.

% and choose $\beta\leq \frac{M_0}{2}$. Then  we get
%$$
%(ii)\leq \bar C_D \int_D g_p(\beta d^s(x)-u(y))\,dy \leq \bar C_D \int_D g_p\left(\frac{M_0}{2}-M_0\right) dy=-g_p\left( \frac{M_0}{2} \right) \bar C_D |D|:=-\widetilde M_0,
%$$
%where $\widetilde M_0>0$.

 Now, define 
$$
M_1:=\inf_{x\in B_1(e_n)\cap B_r^c(0)} u(x)>0, \qquad 
M_2:=\sup_{x\in B_1(e_n)\cap B_r(0)} u(x)>0, 
$$
and we take $r$ small enough  so that $r \in (0, 1/3\sqrt{5})$ and
$$
g_p(M_2) < \dfrac{\widetilde M_0}{\|c\|_{\infty}}.
$$
This can be done since $u$ is continuous in $\overline{\Omega}$ and $u > 0$ in $\Omega$. Observe that this bound is uniform and independent of $r$, although  $r$ depends on the boundary point.
Next, choose 
$$
0 < \beta\leq  \left(\frac{\widetilde M_0-\|c\|_\infty g_p(M_2)}{C_1}\right)^\frac{1}{p-1}
$$
which leads to  (in the weak sense)
\begin{align} \label{eqcomp}
\begin{split}
(-\Delta_p)^s \underline u &\leq \beta^{p-1} C_1   -\widetilde{M}_0 \leq - g_p(M_2) \|c\|_\infty\\
&\leq c g_p(u)\leq (-\Delta_p)^s u \qquad \text{ in } B_1(e_n)\cap B_r(0).
\end{split}
\end{align}
Then, for $x\in B_1^c(e_n)$ we have that
$$
\underline{u}(x) = u(x) \chi_D(x) \leq u(x) 
$$
and for $x\in B_1(e_n)\setminus B_r(0)$ we have that
$$
\underline{u}(x) = \beta d^s(x) \leq \beta \leq u(x).
$$
In sum, we have obtained from the previous expression and  \eqref{eqcomp} that
\begin{align*}
\begin{cases}
(-\Delta_p)^s \underline u  \leq (-\Delta_p)^s u &\quad \text{ weakly in } B_1(e_n)\cap B_r(0),\\
\underline u(x)\leq u(x) &\quad \text{ in } (B_1(e_n)\cap B_r(0))^c.
\end{cases}
\end{align*}
Using the comparison principle given in \cite[Proposition 2.5]{DPQ} gives that
\begin{equation} \label{desuu}
\underline u(x)\leq u(x) \quad \text{ in } B_1(e_n)\cap B_r(0).
\end{equation}
By definition of $d(x)$, for any $t\in (0,1)$
\begin{equation} \label{Des1}
d(te_n)=\delta(te_n)
\end{equation}
where $\delta(x)=\dist(x,\Omega^c)$, and since $te_n \notin D$
\begin{equation} \label{Des2}
\underline u(te_n) = \beta  d^s(te_n),
\end{equation}
this gives, from \eqref{Des1}, \eqref{Des2} and \eqref{desuu}, that
$$
\frac{u(te_n)}{\delta^s(te_n)}=\frac{u(te_n)}{d^s(te_n)}=\frac{u(te_n)}{\beta^{-1}\underline u(te_n)} \geq \beta>0
$$
which completes the proof.
\end{proof}

\section{Hopf's Lemma for sing-changing potentials}\label{hopf lemma sign changing}

Given  a point $x_0\in \partial \Omega$ where the interior ball condition holds, we define inspired by \cite{DSV}, the collection of functions $\mathcal{Z}_{x_0}$ as follows: $u: \mathbb{R}^{n}\to \mathbb{R}$ belongs to $\mathcal{Z}_{x_0}$ if and only if $u$ is continuous in $\mathbb{R}^{n}$, $|u|>0$  in $B_r(x_r)$ for all sufficiently small $r$, $u(x_0)=0$ and the following growth condition is true
\begin{equation}\label{growth boundary u}
\limsup_{r\to 0}\Phi(r)=+\infty
\end{equation}where
$$
\Phi(r):= \dfrac{(\inf_{B_{r/2}(x_r)}|u|)^{p-1}}{r^{p s}}.
$$
We also recall that for $p\geq 2$ we denote $g_p(t):=|t|^{p-2}t$.

\medskip
We state now the main theorem of this section. 

\begin{theorem}\label{Hopf lemma normal derivative}
Let $\Omega \subset \mathbb{R}^n$ be open and bounded and $x_0 \in \partial \Omega$. Assume that $\Omega$ satisfies the interior ball condition at $x_0$.
Let $u\colon \mathbb{R}^n\to \mathbb{R}$ be in $\mathcal{Z}_{x_0}$, such that $u^{-}\in L^{\infty}(\mathbb{R}^{n})$, and
$$
(-\Delta_p)^{s}u \geq c(x)g_p(u)  \quad \text{ in }\Omega$$
weakly, where $c \in L^{1}_{loc}(\Omega)$ with $c^{-}\in L^{\infty}(\Omega)$.
Further, suppose that there is $R>0$ such that $u \geq 0$ in $B_R(x_0)$ and $u > 0 $ in $B_R(x_0)\cap \Omega$. Then, for every $\beta \in (0, \pi/2)$, the following strict inequality holds
\begin{equation}\label{normal derivative at boundary}
\liminf_{\Omega\ni x\to x_0}\dfrac{u(x)-u(x_0)}{|x-x_0|^{s}} > 0,
\end{equation}whenever the angle between $x-x_0$ and the vector joining $x_0$ and the center of the interior ball is smaller than $\pi/2-\beta$. 
\end{theorem}
 
\begin{proof}
For the reader's convenience we split the proof in several steps.

\textbf{Step 1}: First, we analyse the $p-$Laplacian of the distance function to the boundary of the unit ball.  By  Theorem 3.6 in \cite{IMS}, the distance function $d(x)=dist(x, B_1^{c})= 1-|x|$ satisfies that there is $\rho \in (0, 1/2)$ such that

\begin{equation}\label{esti lap distance function}
(-\Delta_p)^{s}d^{s}= f \in L^{\infty}(B_1\setminus \overline{B_{1-\rho}}) \text{ weakly in }B_1\setminus \overline{B_{1-\rho}}.
\end{equation}
Moreover, observe $d^{s}$ also satisfies the lower bound estimate
\begin{equation}\label{lower bound v}
d^{s}(x) \geq \dfrac{1}{2}(1-|x|^{2})^{s}.
\end{equation}We will use this estimate in what follows.

\medskip

\textbf{Step 2}: We prove that there exists a function $\varphi\geq 0$ and a constant $C=C(n,s,p)>0$ such that 
\begin{align} \label{step2'}
\begin{split}
\begin{cases}
(-\Delta_p)^{s} \varphi\leq -1 &\text{ in } x\in B_1\setminus \overline{B_{1-\rho}}, \\
\varphi\geq \frac{1}{2}(1-|x|^2)^s &\text{ in } B_1,\\
\varphi\leq C &\text{ in } B_{1-\rho},\\
\varphi=0 & \text{ in } \mathbb R^n \setminus B_1.
\end{cases}
\end{split}
\end{align}

\medskip

Let $\eta \in C_0^{\infty}(B_{1-2\rho})$ be nonnegative, $\eta \leq 1$, with $\int_{\mathbb{R}^{n}}\eta =1$. Then, for $x \in B_1\setminus \overline{B_{1-\rho}}$, we have
\begin{equation}\label{g lap eta}
\begin{split}
(-\Delta_p)^{s}\eta(x)& = 
\text{p.v.}\, \int_{\mathbb{R}^{n}}g_p\left(\dfrac{\eta(x)-\eta(y)}{|x-y|^{s}} \right)\dfrac{dy}{|x-y|^{n+s}}= -\text{p.v.}\,\int_{\mathbb{R}^{n}}g_p\left(\dfrac{\eta(z+x)}{|z|^{s}} \right)\dfrac{dz}{|z|^{n+s}}\\ & = 
-\text{p.v.}\, \int_{B_{1-2\rho}(-x)}g_p\left(\dfrac{\eta(z+x)}{|z|^{s}} \right)\dfrac{dz}{|z|^{n+s}}.
\end{split}
\end{equation}
Since $z \in B_{1-2\rho}(-x)$, we have that
$$
|z| \leq |z+x|+|x|\leq 2-2\rho,
$$
from where, \eqref{g lap eta} gives that for any $x\in B_1\setminus \overline{B_{1-\rho}}$
\begin{align}\label{est eta 100}
\begin{split}
(-\Delta_p)^{s}\eta(x) 
&\leq -\text{p.v.}\, \int_{B_{1-2\rho}(-x)}g\left(\dfrac{\eta(z+x)}{(2(1-\rho))^s}  \right)\frac{dz}{(2(1-\rho))^{n+s}}\\
&\leq -\left(\dfrac{1}{2(1-\rho)} \right)^{sp+n} \int_{B_{1-2\rho}} g_p(\eta(n(z)))\,dz = -C\left(\dfrac{1}{2(1-\rho)} \right)^{sp+n}. 
\end{split}
\end{align}
Therefore, for a  constant $C>0$ large enough to be chosen, we define
$$
\varphi= d^{s}+C\eta.
$$
Employing Lemma 2.8 in  \cite{IMS}, we get, weakly in $B_1\setminus \overline{B_{1-\rho}}$ that
\begin{equation}\label{g lap phi}
(-\Delta_p)^{s}\varphi= (-\Delta_p)^{s}d^{s}  + h,
\end{equation}where the function $h$ is given by
\begin{equation*}
\begin{split}
h(x)&= 2 \int_{B_{1-2\rho}}\left[g_p\left(\dfrac{d^{s}(x)-d^{s}(y)-C\eta(y)}{|x-y|^{s}} \right)- g_p\left(\dfrac{ d^{s}(x)- d^{s}(y)}{|x-y|^{s}} \right) \right]\dfrac{dy}{|x-y|^{n+s}}\\& \leq  2c_1\int_{B_{1-2\rho}}g_p\left(\dfrac{C(\eta(x)-\eta(y))}{|x-y|^{s}} \right)\dfrac{dy}{|x-y|^{n+s}}
\end{split}
\end{equation*}where we have used the following  inequality  
%(see for instance Lemma 2.1 in \cite{MSV21})
$$
g_p(b)-g_p(a)\leq c_1g_p(b-a), \quad b \leq a,
$$
being $c_1 > 0$. Hence, by \eqref{est eta 100}, we have pointwisely
\begin{equation}
h(x)\leq 2c_1C^{p-1}(-\Delta_p)^{s}\eta(x) \leq -2c_1C^{p-1}\left(\dfrac{1}{2(1-\rho)} \right)^{sp+n}.
\end{equation}Therefore, choosing $C>0$ large enough and recalling \eqref{esti lap distance function}, we obtain weakly in $ B_1\setminus \overline{B_{1-\rho}}$ that
\begin{equation}\label{estimate negative c}
(-\Delta_p)^{s}\varphi \leq -1.
\end{equation}
Also, observe that $\varphi \leq 1 + C$, $\varphi=0$ in $\mathbb{R}^{n}\setminus B_1$ and moreover by \eqref{lower bound v},
\begin{equation*}
\varphi(x) \geq  d^{s}(x) \geq \dfrac{1}{2 }(1-|x|^{2})^{s}.
\end{equation*}

\medskip
\textbf{Step 3}:  Consider the  balls $B_{r}(x_r)$ and points $x_r$ from the interior ball condition for $\Omega$ at $x_0$. For each $r > 0$ small enough, let for any  $x \in B_r(x_r)$
$$
\alpha_r :=\dfrac{1}{C}\inf_{B_{r(1-\rho)}(x_r)}u, 
\qquad 
\psi_r(x):=\alpha_r\varphi\left(\dfrac{x-x_r}{r}\right).
$$Here, the constant $C$ is the one defined in \eqref{step2'}. 
Then under these assumptions we will prove that it holds that
\begin{align} \label{step3'}
\begin{split}
\begin{cases}
(-\Delta_p)^{s}\psi_r \leq - \dfrac{1}{r^{s}} \left( \dfrac{\alpha_r}{r^{s}}\right)^{p-1} & \text{weakly in } B_r(x_r)\setminus \overline{B_{r(1-\rho)}(x_r)},\\
\psi_r \leq \alpha_r C &\text{ in }B_r(x_r), \\
\psi_r= 0 &\text{ in }\mathbb{R}^{n}\setminus B_r(x_r),\\ 
\psi_r\geq  \dfrac{\alpha_r}{2r^{2s}}(r^{2}-|x-x_r|^{2})^{s} &\text{ in } B_r(x_r).
\end{cases}
\end{split}
\end{align}

\medskip 
To prove \eqref{step3'}, we will proceed as follows: observe that for $x \in B_r(x_r)\setminus \overline{B_{r(1-\rho)}(x_r)}$ we have that $x'=\frac{x-x_r}{r} \in B_1 \setminus \overline{B_{1-\rho}}$. Then for any $\phi \in W_0^{s, p}(B_r(x_r)\setminus \overline{B_{r(1-\rho)}(x_r)})$,

\begin{equation}\label{g lapp psi r}
\begin{split}
\left\langle (-\Delta_p)^s \psi_r, \phi \right\rangle & = \int_{\mathbb{R}^n}\int_{\mathbb{R}^n}g_p\left(\dfrac{\psi_r(x)-\psi_r(y)}{|x-y|^s}\right)\dfrac{\phi(x)-\phi(y)}{|x-y|^s}\dfrac{dx\,dy}{|x-y|^n}\\ & = \int_{\mathbb{R}^n}\int_{\mathbb{R}^n}g_p\left(\alpha_r \frac{\varphi\left(\frac{x-x_r}{r}\right)-\varphi\left(\frac{y-x_r}{r}\right)}{|x-y|^{s}}\right)\dfrac{\tilde{\phi}\left(\frac{x-x_r}{r}\right)-\tilde{\phi}\left(\frac{y-x_r}{r}\right)}{|x-y|^s}\dfrac{dx\,dy}{|x-y|^n} \\ & = \dfrac{\alpha^{p-1}_r}{r^{sp}} \int_{\mathbb{R}^n}\int_{\mathbb{R}^n}g_p\left( \dfrac{\varphi(x')-\varphi(z)}{|x'-z|^{s}}\right)\dfrac{\tilde{\phi}\left(x'\right)-\tilde{\phi}\left(y'\right)}{|x'-y'|^s}\dfrac{dx'\,dy'}{|x'-y'|^n}\\ & = \dfrac{\alpha^{p-1}_r}{r^{sp}}\left\langle (-\Delta_p)^s \varphi, \tilde{\phi} \right\rangle ,
\end{split}
\end{equation}where
$$\tilde{\phi}(x'):= \phi(rx'+x_r)= \phi(x).$$Observe that $\tilde{\phi}\in W_0^{s, p}( B_1 \setminus \overline{B_{1-\rho}})$ and $\tilde{\phi}\geq 0$. Hence, by \eqref{estimate negative c},
\begin{equation}\label{g lap psi r}
 (-\Delta_p)^s \psi_r \leq -\dfrac{\alpha^{p-1}_r}{r^{sp}},
\end{equation}weakly in $B_r(x_r)\setminus \overline{B_{r(1-\rho)}(x_r)}$. Moreover, scaling and using again \eqref{step2'} we get that
\begin{equation}
\psi_r \leq \alpha_r C \quad \text{in }B_r(x_r), 
\quad 
\psi_r= 0 \quad \text{in }\mathbb{R}^{n}\setminus B_r(x_r), 
\text{ and }
\psi_r\geq  \dfrac{\alpha_r}{2r^{2s}}(r^{2}-|x-x_r|^{2})^{s}.
\end{equation}

\medskip
\textbf{Step 4}: 
The function $w:= \psi_r-u^{-}$ satisfies that $u \geq w$ a.e. in $\mathbb R^n$.

\medskip

To prove this assertion we will use comparison. Observe first that $w \leq u$ in $\mathbb{R}^{n}\setminus B_r(x_r)$, moreover, in $B_{r(1-\rho)}(x_r)$ we have that
$$
w = \psi_r \leq \alpha_rC \leq \inf_{B_{r(1-\rho)}(x_r)}u \leq u.
$$
Also,  by \eqref{step3'} and Lemma 2.8 in  \cite{IMS},
$$(-\Delta_p)^{s} w \leq - \dfrac{1}{r^{s}} \left( \dfrac{\alpha_r}{r^{s}}\right)^{p-1} + h,$$weakly in $B_r(x_r)\setminus \overline{B_{r(1-\rho)}(x_r)}$, where the function $h$ is given by
\begin{equation*}
\begin{split}
h(x)& = 2 \int_{supp\,\,u^{-}}\left[g_p\left(\dfrac{\psi_r(x)-\psi_r(y)+u^{-}(y)}{|x-y|^{s}}\right)-g_p\left(\dfrac{\psi_r(x)-\psi_r(y)}{|x-y|^{s}} \right) \right]\dfrac{dy}{|x-y|^{n+s}}\\ & =  2 \int_{supp\,\,u^{-}}\left[g_p\left(\dfrac{\psi_r(x)+u^{-}(y)}{|x-y|^{s}}\right)-g_p\left(\dfrac{\psi_r(x)}{|x-y|^{s}} \right) \right]\dfrac{dy}{|x-y|^{n+s}}
 \leq C^{*}
\end{split}
\end{equation*}since $\psi_r$ and $u^{-}$ are bounded and $|x-y|\geq \delta >0$, for some $\delta$. Hence, weakly in $B_r(x_r)\setminus \overline{B_{r(1-\rho)}(x_r)}$, we obtain
$$
(-\Delta_p)^{s} w \leq  C^{*} -\dfrac{\Phi(r)}{C^{p^{+}-1}}  = C^{*}- \dfrac{1}{r^{s}} \left( \dfrac{\alpha_r}{r^{s}}\right)^{p-1} \leq 
-\|c^{-}u^{+}\|_{L^{\infty}(\mathbb{R}^{n})} \leq cg_p(u) \leq  (-\Delta_p)^{s}u,
$$
where we have used \eqref{growth boundary u}. By the comparison principle, we obtain that $u \geq w$ in $\mathbb{R}^{n}$.

\medskip
\textbf{Step 5}: Final argument.

\medskip
To complete the proof, we argue as in \cite{DSV}. We include details for completeness. For $\beta \in (0, \pi/2)$, we consider the set 
\begin{equation}
\mathcal{C}_\beta:=\left\lbrace x\in \Omega: \dfrac{x-x_0}{|x-x_0|}\cdot \nu > c_\beta \right\rbrace,
\end{equation}where $\nu$ is the inward normal vector joining $x_0$ with the center of the interior ball, and define the constant $c_\beta := \cos\left(\frac{\pi}{2}-\beta \right) > 0$. Take any sequence of points $x_k\in \mathcal{C}_\beta$ such that $x_k \to x_0$. Then,
\begin{equation*}
\begin{split}
|x_k-x_r|^{2}& = |x_k -x_0-r\nu|^{2}= |x_k-x_0|^{2}+r^{2} -2r(x_k-x_0)\cdot \nu \\ & \leq r^{2}-|x_k-x_0|(2c_\beta r - |x_k-x_0|) < r^{2},
\end{split}
\end{equation*}for $k$ large enough. Moreover, since
$$|x_k-x_r| \geq |x_r-x_0|+|x_k-x_0| = r-|x_k-x_0| > r(1-\rho)$$for $k$ large, then $x_k \in B_r(x_r)\setminus \overline{B_{r(1-\rho)}(x_r)}$. Next,
\begin{equation*}
\begin{split}
u(x_k)& \geq w(x_k)= \psi_r(x_k) \geq \frac{\alpha_r}{r^{2s}2^{s}}(r^{2}-|x_k-x_r|^{2})_{+}^{s}\\
& = \frac{\alpha_r}{2r^{2s}}( r^{2}-|x_k-x_0-r\nu|^{2})_{+}^{s}\\ & = \frac{\alpha_r}{2r^{2s}}(2(x_k-x_0)\cdot \nu -|x_k-x_0|^{2})_{+}^{s}\geq  \frac{\alpha_r}{2r^{2s}}(2c_\beta r|x_k-x_0|-|x_k-x_0|^{2})_{+}^{s}.
\end{split}
\end{equation*}Therefore,
$$
\liminf_{k \to \infty}\dfrac{u(x_k)-u(x_0)}{|x_k-x_0|^{s}}\geq  \frac{\alpha_r}{2r^{2s}}\liminf_{k \to \infty}(2c_\beta r-|x_k-x_0|)_{+}^{s} = \dfrac{2^{s-1}\alpha_rc_\beta^{s}}{r^{2s}}.
$$
This ends the proof of the theorem. 
\end{proof}

\begin{remark}\label{remark boundary}
Observe that  although the equation is nonlocal, in the above proof, we only used that $u$ is a supersolution near the boundary of $\Omega$ (see Step 4). 
\end{remark}

By the interior sphere condition and \cite{HKS}, for any $x_0 \in \partial \Omega$, there exist $x_1\in \Omega$ and $\rho > 0$ such that
$$B_{\rho}(x_1)\subset \Omega, \quad d(x)=|x-x_0|,$$for all $x\in \Omega$ of the form
$$x= tx_1+(1-t)x_0, \quad t \in [0, 1].$$Hence, by the interior ball condition, we can find a sequence $x_n \in \Omega$ such that
$$\delta(x_n)= \text{dist}\,(x_n, \partial \Omega)= |x_n-x_0| \to 0 \text{ as }n \to \infty.$$

Moreover, if $u/\delta^{s}$ can be extended to a continuous function to $\partial \Omega$, Theorem \ref{Hopf lemma normal derivative} implies that
$$\dfrac{u(x_0)}{\delta^{s}(x_0)}
= \lim_{\overline\Omega \ni x\to x_0}\dfrac{u(x)-u(x_0)}{\delta^{s}(x)}
= \lim_{n\to \infty}\dfrac{u(x_n)-u(x_0)}{\delta^{s}(x_n)}
= \lim_{n\to \infty}\dfrac{u(x_n)-u(x_0)}{|x_n-x|^{s}}>0.$$

We establish another direct consequence. Since the function $g_p$ is odd, by changing $u$ by $-u$ in Theorem \ref{Hopf lemma normal derivative}, we obtain the following corollary.

\begin{corollary}\label{corolario normal frac}
Let $\Omega \subset \mathbb{R}^n$ be open and bounded and $x_0 \in \partial \Omega$. Assume that $\Omega$ satisfies the interior ball condition at $x_0$.
Let $u\colon\mathbb{R}^n\to \mathbb{R}$ be in $\mathcal{Z}_{x_0}$, such that $u^{+}\in L^{\infty}(\mathbb{R}^{n})$, and
$$
(-\Delta_p)^{s}u \leq c(x)g_p(u)  \quad \text{weakly in }\Omega
$$
where $c \in L^{1}_{loc}(\Omega)$ with $c^{-}\in L^{\infty}(\Omega)$.
Further, suppose that there is $R>0$ such that $u \leq 0$ in $B_R(x_0)$, $u < 0 $ in $B_R(x_0)\cap \Omega$. Then, for every $\beta \in (0, \pi/2)$, the following strict inequality holds
\begin{equation}\label{normal derivative at boundary}
\limsup_{x\in \Omega, x\to x_0}\dfrac{u(x)-u(x_0)}{|x-x_0|^{s}} < 0,
\end{equation}whenever the angle between $x-x_0$ and the vector joining $x_0$ and the center of the interior ball is smaller than $\pi/2-\beta$. 
\end{corollary}

Reasoning as before, we may also conclude, under the assumptions of Corollary \ref{corolario normal frac} and the hypothesis that $u/\delta^{s} \in C(\overline{\Omega})$, that
$$\dfrac{u(x_0)}{\delta^{s}(x_0)} < 0.$$

\section{Boundary point theorems}\label{bp theorem}

As applications of Theorem \ref{Hopf lemma normal derivative}, we will provide for two version of the classical boundary point lemma in the setting of the fractional $p-$Laplacian (see for instance Section 2.7 in \cite{PS} for elliptic equations in the local case).

The first result is a direct consequence of Theorem \ref{Hopf lemma normal derivative} and  Corollary \ref{corolario normal frac}.
\begin{theorem}\label{bpl1}
Let $\Omega \subset \mathbb{R}^n$ be open and bounded and $x_0 \in \partial \Omega$. Assume that $\Omega$ satisfies the interior ball condition at $x_0$.
Let $u, v:\mathbb{R}^n\to \mathbb{R}$ be in $\mathcal{Z}_{x_0}$, such that $u^{+}, v^{-}\in L^{\infty}(\mathbb{R}^{n})$, and
$$(-\Delta_p)^{s}u \leq c(x)g_p(u), \quad (-\Delta_p)^{s}v \geq c(x)g_p(v) \quad \text{weakly in }\Omega,$$
where $V \in L^{1}_{loc}(\Omega)$ with $c^{-}\in L^{\infty}(\Omega)$.
Further, suppose that there is $R>0$ such that $u \leq 0 \leq v$ in $B_R(x_0)$, $u < 0 < v$ in $B_R(x_0)\cap \Omega$. Then, for every $\beta \in (0, \pi/2)$, we have the following strict inequality
\begin{equation}
\dfrac{u(x)-u(x_0)}{|x-x_0|^{s}}  < \dfrac{v(x)-v(x_0)}{|x-x_0|^{s}}, \quad \text{as }x \to x_0, \, x \in \Omega,
\end{equation}whenever the angle between $x-x_0$ and the vector joining $x_0$ and the center of the interior ball is smaller than $\pi/2-\beta$. 
\end{theorem}

The following boundary point theorem does not require a constant sign of the solutions in a neighborhood of the boundary point. Instead, we will need a  natural decaying (thanks to Hopf's Lemma)   of the difference $v-u$ near boundary points. Moreover, we will also impose more regularity on the solutions. To motive this latter regularity hypothesis, we first prove that it implies that the operator $(-\Delta_p)^{s}$ is continuous and hence point-wisely defined. This result is a refinement of \cite[Lemma 2.17]{FBPLS} where a similar result has been obtained in the context of Orlicz functions. Observe that our proof also holds in that framework.

\begin{lemma} Suppose $u \in L_{s p}(\mathbb{R}^{n})\cap C^{1, \gamma}_{loc}(\mathbb{R}^{n})$, for some $\gamma \geq \max\{0, 1-p(1-s)\} $. Then
$$
(-\Delta_p)^{s}u(x)< \infty \quad \text{a.e. in }\mathbb{R}^{n}.
$$
Moreover, if $u \in L^{\infty}(\mathbb{R}^{n})\cap C^{1}(\mathbb{R}^{n})$ and
$$
p> \dfrac{1}{1-s},
$$
then $(-\Delta_p)^{s}u \in C(\mathbb{R}^{n})$.
\end{lemma}
\begin{proof}Let
$$
h(x):=\int_{\mathbb{R}^{n}\setminus B_1(x)}g_p\left(\dfrac{u(x)-u(y)}{|x-y|^{s}} \right)\dfrac{dy}{|x-y|^{n+s}}
$$
and for $\varepsilon > 0$ define 
$$
h_\varepsilon(x):=\int_{B_1(x)}g_p\left(\dfrac{u(x)-u(y)}{|x-y|^{s}} \right)\chi_{x, \varepsilon}(y)\dfrac{dy}{|x-y|^{n+s}},
$$
where $\chi_{x, \varepsilon}$ is the characteristic function of the set $B_1(x)\setminus B_{\varepsilon}(x)$.  By Lemma A.5 in \cite{FBSV2}, there is a constant $C> 0$ such that
\begin{equation}\label{contant lemma}
|x-y|\geq C (1+|y|), \text{ for all }y \in \mathbb{R}^{n}\setminus B_1(x).
\end{equation}Hence,
\begin{equation*}
|h(x)|\leq C_0 \int_{\mathbb{R}^{n}\setminus B_1(x)} \left( g_p\left(\dfrac{Cu(x)}{(1+|y|)^{s}} \right) +g_p\left(\dfrac{Cu(y)}{(1+|y|)^{s}} \right) \right)\dfrac{dy}{(1+|y|)^{n+s}}< \infty,
\end{equation*}
where we have used the facts that the constant function $C|u(x)|$ and the function $u(y)$ are in $L_{sp}(\mathbb{R}^{n})$. Regarding $h_{\varepsilon}(x)$, we write
$$
h_{\varepsilon}(x)= \dfrac{1}{2}\int_{B_1(x)}g_p\left(\dfrac{u(x)-u(y)}{|x-y|^{s}} \right)\chi_{x, \varepsilon}(y)\dfrac{dy}{|x-y|^{n+s}}+\dfrac{1}{2}\int_{B_1(x)}g_p\left(\dfrac{u(x)-u(y)}{|x-y|^{s}} \right)\chi_{x, \varepsilon}(y)\dfrac{dy}{|x-y|^{n+s}}
$$
and we make the change of variables $z=y-x$ in the first integral and $z=x-y$ in the second, to get
\begin{equation}\label{ineq h epsilon}
\begin{split}
h_{\varepsilon}(x)&= \dfrac{1}{2}\int_{B_1}g_p\left(\dfrac{u(x)-u(z+y)}{|z|^{s}} \right)\chi_{\varepsilon}(z)\dfrac{dz}{|z|^{n+s}}+\dfrac{1}{2}\int_{B_1}g_p\left(\dfrac{u(x)-u(x-z)}{|z|^{s}} \right)\chi_{\varepsilon}(z)\dfrac{dz}{|z|^{n+s}}\\ & =\dfrac{1}{2}\int_{B_1}\left[g_p\left(\dfrac{u(x)-u(z+y)}{|z|^{s}} \right)-g_p\left(\dfrac{u(x-z)-u(x)}{|z|^{s}} \right)\right]\chi_{\varepsilon}(z)\dfrac{dz}{|z|^{n+s}},
\end{split}
\end{equation}where $\chi_{\varepsilon}$ is the characteristic function of $B_1\setminus B_{\varepsilon}$. Due to the inequality
\begin{equation}\label{ineq with g}
|g_p(a+b)-g_p(b)|\leq C_1(|b|+|a|)^{p-2}|a|,
\end{equation}we obtain from  \eqref{ineq h epsilon} that
\begin{equation}\label{est integrand}
\begin{split}
& \bigg |g_p\left(\dfrac{u(x)-u(z+y)}{|z|^{s}} \right)-g_p\left(\dfrac{u(x-z)-u(x)}{|z|^{s}}\right)\bigg| \\ & \leq   \left(\dfrac{|u(x)-u(x-z)|}{|z|^{s}}+\dfrac{|2u(x)-u(x+z)-u(x-z)|}{|z|^{s}} \right)^{p-2}\dfrac{|2u(x)-u(x+z)-u(x-z)|}{|z|^{s}}.
\end{split}
\end{equation}Since  $$|2u(x)-u(x+z)-u(x-z)| \leq C|h|^{1+\gamma},$$ 
we obtain that the integrand \eqref{ineq h epsilon} is bounded from above for any $\varepsilon$ by
\begin{equation}
|z|^{(1-s)(p-2)+1+\gamma-n-2s},
\end{equation}which  is integrable in $B_1$  provided
$$\gamma \geq \max\{0, 1-p(1-s)\}.$$
Hence, by dominated convergence theorem,
$$
(-\Delta_p)^{s}u(x)= h(x)+\lim_{\varepsilon\to 0^{+}}h_\varepsilon(x) < \infty,
$$
for a.e. $x$. This ends with the proof of the first assertion.

We next prove the continuity of $(-\Delta_p)^{s}u$. We denote, for $\varepsilon>0$
$$
(-\Delta_{p,\varepsilon}) u(x)=\int_{\mathbb R^n \setminus B_\varepsilon(x)} g_p\left(\dfrac{u(x)-u(z)}{|x-z|^{s}} \right)\dfrac{dz}{|x-z|^{n+s}}
$$
and write
\begin{equation}\label{continuity 0}
\begin{split}
&|(-\Delta_{p, \varepsilon})^{s}u(x)-(-\Delta_{p, \varepsilon})^{s}u(y)| \leq \\& \quad  \leq \bigg| \int_{\mathbb{R}^{n}\setminus B_1(x)}g_p\left(\dfrac{u(x)-u(z)}{|x-z|^{s}} \right)\dfrac{dz}{|x-z|^{n+s}}- \int_{\mathbb{R}^{n}\setminus B_1(y)}g_p\left(\dfrac{u(y)-u(z)}{|y-z|^{s}} \right)\dfrac{dz}{|y-z|^{n+s}} \bigg|\\ & \quad + \bigg| \int_{ B_1(x)\setminus B_\varepsilon(x)}g_p\left(\dfrac{u(x)-u(z)}{|x-z|^{s}} \right)\dfrac{dz}{|x-z|^{n+s}}- \int_{ B_1(y)\setminus B_\varepsilon(y)}g_p\left(\dfrac{u(y)-u(z)}{|y-z|^{s}} \right)\dfrac{dz}{|y-z|^{n+s}} \bigg|= I_1 +I_2.
\end{split}
\end{equation}By changing variables and using that $u \in L^{\infty}(\mathbb{R}^{n})$ and \eqref{ineq with g}, we get
\begin{equation}\label{continuity 1}
\begin{split}
& I_1 = \bigg| \int_{\mathbb{R}^{n}\setminus B_1}g_p\left(\dfrac{u(x)-u(x+h)}{|h|^{s}} \right)\dfrac{dh}{|h|^{n+s}}- \int_{\mathbb{R}^{n}\setminus B_1}g_p\left(\dfrac{u(y)-u(y+h)}{|h|^{s}} \right)\dfrac{dh}{|h|^{n+s}} \bigg|\\ & \,\, \leq \int  _{\mathbb{R}^{n}\setminus B_1}\bigg(\dfrac{|u(y)-u(y+h)|}{|h|^{s}} + \dfrac{|u(x)-u(x+h)-u(y)+u(y+h)|}{|h|^{s}} \bigg)^{p-2}\dfrac{|u(x)-u(x+h)-u(y)+u(y+h)|}{|h|^{s}}\dfrac{dh}{|h|^{n+s}} \\ & \,\, \leq |x-y| \int  _{\mathbb{R}^{n}\setminus B_1}|h|^{-s(p-2)-2s-n}dh \leq C|x-y|.
\end{split}
\end{equation}
Regarding $I_2$, we take any $\alpha \in (0, 1)$ and we proceed as follows, using again \eqref{ineq with g} and that $u \in C^{1}(\mathbb{R}^{n})$:
\begin{equation}\label{continuity 2}
\begin{split}
& I_2 = \bigg| \int_{ B_1\setminus B_\varepsilon}g_p\left(\dfrac{u(x)-u(x+h)}{|h|^{s}} \right)\dfrac{dh}{|h|^{n+s}}- \int_{\mathbb{R}^{n}\setminus B_1}g_p\left(\dfrac{u(y)-u(y+h)}{|h|^{s}} \right)\dfrac{dh}{|h|^{n+s}} \bigg|\\ & \, \leq \int  _{ B_1\setminus B_\varepsilon} \bigg(\dfrac{|u(y)-u(y+h)|}{|h|^{s}} + \dfrac{|u(x)-u(x+h)-u(y)+u(y+h)|}{|h|^{s}} \bigg)^{p-2}\\& \qquad \times \dfrac{|u(x)-u(x+h)-u(y)+u(y+h)|^{\alpha+1-\alpha}}{|h|^{s}}\dfrac{dh}{|h|^{n+s}} \\ & \, \leq C|x-y|^{1-\alpha} \int_{ B_1\setminus B_\varepsilon}|h|^{(1-s)(p-2)+\alpha-2s-n}dh.
\end{split}
\end{equation}Then, observe that 
$$
\int_{ B_1\setminus B_\varepsilon}|h|^{(1-s)(p-2)+\alpha-2s-n}dh<\infty \quad \text{ provided that } \quad p\geq \dfrac{2-\alpha}{1-s}.
$$
Since $\alpha$ is arbitrarily chosen in $(0, 1)$, the integral is finite provided
$p> \dfrac{1}{1-s}$. Therefore, combining \eqref{continuity 1} and \eqref{continuity 2} with \eqref{continuity 0}, and taking $\varepsilon\to 0^{+}$,  we conclude that $(-\Delta_p)^{s}u \in C(\mathbb{R}^{n})$. 
\end{proof}

We next give the main boundary point result for the fractional $p-$Laplacian. We state it in a ball for simplicity.

\begin{theorem}\label{bpl2}Assume  that $p> \min\{1/(1-s), 2\}$. Let $u, v \in C^{1}(\overline{B_1})\cap L^{\infty}(\mathbb{R}^{n})$ be in $\mathcal{Z}_{x_0}$ at any boundary point $x_0$. Moreover, suppose that
\begin{equation}\begin{cases}\label{assumpt bpl}
(-\Delta_p)^{s}u-c(x)g_p(u) \leq (-\Delta_p)^{s}v-c(x)g_p(v) \quad \text{ weakly in }B_1\\ u < v  \quad \text{ in }B_1\\ v= u \quad \text{ in }\mathbb{R}^{n}\setminus B_1\\ v(x)-u(x)\geq C(1-|x|)^{s}, \quad \text{ uniformly as }|x|\to 1.
\end{cases}
\end{equation}Assume that  $c\in L^{\infty}(B_1)$. Then, for any $x_0\in \partial B_1$,
\begin{equation}\label{conclusion bpt}
\dfrac{u(x)-u(x_0)}{|x-x_0|^{s}}  < \dfrac{v(x)-v(x_0)}{|x-x_0|^{s}}, \quad \text{as }x \to x_0 \quad (x \in  B_1).
\end{equation}
\end{theorem}

\begin{proof}
Observe that $w:=v-u$ is positive in $B_1$ and $w=0$ on $\partial B_1$. Moreover, $w \in \mathcal{Z}_{x_0}$ at any boundary point $x_0$. Observe that the conclusion \eqref{conclusion bpt} is obtained by applying Theorem \ref{Hopf lemma normal derivative} if we prove that
\begin{equation}\label{ineq with w}
(-\Delta_p)^{s}w-c(x)g_p(w) \geq 0 
\end{equation}
weakly in a neighborhood of the boundary of $B_1$ (see Remark \ref{remark boundary}). Now, to establish \eqref{ineq with w} fix $\phi \in W_0^{s, p}(B_1\setminus B_{1-\delta})$, $\phi \geq 0$, with $\delta > 0$ to be chosen later, and compute

\begin{equation}\label{est 2 w}
\begin{split}
\left\langle (\Delta_p)^{s}w -cg_p(w), \phi \right\rangle & \geq  2\int_{\mathbb{R}^n\setminus B_1}\int_{B_1}g_p\left(\dfrac{v(x)-u(x)}{|x-y|^s} \right)\phi(x)\dfrac{dx\,dy}{|x-y|^{n+s}} \\ &  + \int_{B_1}\int_{B_1}g_p\left(\dfrac{v(x)-v(y)-u(x)+u(y)}{|x-y|^s} \right)\dfrac{\phi(x)-\phi(y)}{|x-y|^s}\dfrac{dx\,dy}{|x-y|^{n}}\\ & \qquad  -\int_{B_1}c^{+}(x)g_p(w(x))\phi(x)\,dx\\ & = 2\int_{ B_1}\left[\int_{\mathbb{R}^n\setminus B_1}g_p\left(\dfrac{v(x)-u(x)}{|x-y|^s} \right)\dfrac{dy}{|x-y|^{n+s}}\right]\phi(x)dx\\ & +  2\int_{B_1}\left[\int_{B_1}g_p\left(\dfrac{v(x)-v(y)-u(x)+u(y)}{|x-y|^s} \right)\dfrac{dy}{|x-y|^{n+s}}\right]\phi(x)\,dx\\ & -\int_{B_1}c^{+}(x)g_p(w(x))\phi(x)\,dx \\ &   \geq  I_1+I_2 -\int_{B_1}c^{+}(x)g_p(w(x))\phi(x)\,dx.
\end{split}
\end{equation}
%\begin{equation}
%\begin{split}
%&(-\Delta_p)^{s}w(x)-c(x)g_p(w) \\& \qquad   \geq \int_{\mathbb{R}^{n}\setminus B_1}g_p\left(\dfrac{v(x)-u(x)}{|x-y|^{s}} \right)\dfrac{dy}{|x-y|^{s+n}} + \int_{B_1}g_p\left(\dfrac{v(x)-u(x)-v(y)+u(y)}{|x-y|^{s}} \right)\dfrac{dy}{|x-y|^{s+n}} \\ & \qquad  -c^{+}(x)g_p(v(x)-u(x)) \\ & \qquad  \geq
%\end{split}
%\end{equation}
We first work on $I_2$. Since $w \in C^{1}(\overline{\Omega})$,
\begin{equation}\label{est w 0}
\begin{split}
|I_2|\leq C\int_{B_1}\left[\int_{B_1}g_p(|x-y|^{1-s})\dfrac{dy}{|x-y|^{n+s}}\right] \phi(x)\,dx \leq C\int_{B_1}\phi(x)\,dx,
\end{split}
\end{equation}since $p> 1/(1-s)$.  Regarding $I_1$,
\begin{equation}\label{est 1 w}
\begin{split}
I_1 & \geq  \int_{B_1\setminus B_{1-\delta}}\left[\int_{\{y \in \mathbb{R}^{n}\setminus B_1: 1-|x| \leq |y-x|\leq 2(1-|x|) \}} g_p\left(\dfrac{v(x)-u(x)}{|x-y|^{s}} \right)\dfrac{dy}{|x-y|^{s+n}}\right]\phi(x)\,dx\\ & \geq \int_{B_1\setminus B_{1-\delta}}\left[\int_{\{y \in \mathbb{R}^{n}\setminus B_1: 1-|x| \leq |y-x|\leq 2(1-|x|) \}} g_p\left(\dfrac{v(x)-u(x)}{2^{s}(1-|x|)^{s}} \right)\dfrac{dy}{|x-y|^{s+n}}\right]\phi(x)\,dx\\& \geq g_p(C)\int_{B_1\setminus B_{1-\delta}}\left[\int_{\{y \in \mathbb{R}^{n}\setminus B_1: 1-|x| \leq |y-x|\leq 2(1-|x|) \}}|x-y|^{-s-n}\,dy\right]\phi(x)\,dx \qquad (\text{by }\eqref{assumpt bpl})\\ &  \geq \dfrac{g_p(C)}{s}\int_{B_1\setminus B_{1-\delta}}\left[(1-|x|)^{-s}-2^{-s}(1-|x|)^{-s}\right]\phi(x)\,dx\\ & \geq  \dfrac{g_p(C)}{s}\left(1-\dfrac{1}{2^{s}}\right) \delta^{-s}\int_{B_1\setminus B_{1-\delta}}\phi(x)\,dx.
\end{split}
\end{equation}Finally, the integral
$$-\int_{B_1}c^{+}(x)g_p(w(x))\phi(x)\,dx$$is clearly bounded from below by
$$C\int_{B_1}\phi(x)\,dx,$$for some $C$. Hence, by \eqref{est 2 w}, \eqref{est w 0} and \eqref{est 1 w}, and choosing $\delta > 0$ small enough, we prove that $w$ satisfies \eqref{ineq with w} weakly in $B_1\setminus B_{1-\delta}$. This ends the proof of the theorem.
\end{proof}

\begin{remark}
In view of \eqref{est 1 w}, the lower decay of $v(x)-u(x)$ near the boundary may be relaxed to 
$$v(x)-u(x)\geq c(1-|x|)^{s+\varepsilon},$$where $\varepsilon > 0$ and satisfies
$$t^{(p-1)\varepsilon-s}\to \infty\quad \text{ as }t\to 0^{+}.$$\end{remark}

\section{An overdetermined problem}\label{overdetermined}

Given $p\geq 2$, we consider  the following overdetermined problem
\begin{align}\label{over}
\begin{cases}
(-\Delta_p)^{s} u=1 & \text{ in } \Omega\\
u=0 &\text{ in } \mathbb R^n\setminus \Omega\\
\lim_{\Omega\ni z\to z} \frac{u(x)}{\delta_\Omega(x)^s}=q(|z|) &\text{ for every }z\in \partial\Omega, 
\end{cases}
\end{align}
where $q\colon \partial \Omega\to\mathbb{R}$ is a suitable function. Suppose that \eqref{over} is solvable, that is, there exists $u\in C(\mathbb R^n)$ such that the ratio $u(x)/(\delta_\Omega(x))^s$ has a continuous extension to $\overline\Omega$ and the three conditions prescribed in \eqref{over} are satisfied. We will answer the following question: is it possible to infer that $\Omega$ is a ball?

\medskip
In the linear case the answer is positive and it is given in \cite[Theorem 1.3]{GS} when assuming $\Omega$ to satisfy the interior ball condition on $\partial \Omega$, and $q(r)$ to satisfy that $q(r)/r^s$ is strictly increasing in $r>0$.

Understanding the behavior of the torsion problem is crucial for providing a response in the nonlinear setting:   in Lemma 4.1 of  \cite{IMS} it is proved that there is a unique weak solution $w \in W^{s,p}_0(\Omega)$ of the   problem
\begin{align}\label{torsion}
\begin{cases}
(-\Delta_p)^{s} w=1 & \text{ in } B_1\\
w=0 & \text{ in } \mathbb R^n\setminus \Omega
\end{cases}
\end{align}
which is bounded, radially symmetric, non increasing and positive, i.e., $w(x)=\omega(|x|)$ for some $\omega\colon \mathbb R_+\to \mathbb R$. More precisely, by Lemma 2.9 in \cite{IMS}, if we define $u_R(x)=w(|x|/R)$ then the scaled function $u_R$ solves
\begin{align}\label{torsionR}
\begin{cases}
(-\Delta_{p})^s u_R=R^{-sp} & \text{ in } B_R\\
u_R=0 & \text{ in } \mathbb R^n\setminus B_R.
\end{cases}
\end{align}

In particular, since $\delta_R(x)= \text{dist }(x, \partial B_R)=R-|x|$, we can define for any $z\in \partial B_R$ the function
\begin{equation}\label{rho s}
\lim_{B_R\ni x\to z} \frac{u_R(x)}{(\delta_{B_R}(x))^s}= \lim_{B_R\ni x\to z} \frac{\omega(|x|/R)}{(R-|x|)^s}:=\rho_s(R).
\end{equation}

This boundary value may be thought   as a fractional replacement of the inner normal derivative in the local case. 

In the case $p=2$, it is known (see for instance \cite{GS}) that $u_R(x)=\gamma_{n,s}((R^2-|x|^2)_+)^s$, where $\gamma_{n,s}$ is a positive constant depending only of $n$ and $s$, so, in this case, formula \eqref{rho s}  gives that $\rho_s(R)=2^s\gamma_{n,s}R^s$.

 \medskip

\begin{theorem}  \label{overth}
Let $\Omega$  be a bounded open set in $\mathbb R^n$, $n\geq 1$, containing the origin and satisfying the interior ball condition at any $z\in \partial \Omega$, and let $q(r)$ be a non-negative function of $r>0$. Assume that the ratio
$$
q(r)/ \rho_s(r)
$$
is strictly increasing in $r>0$. Then if \eqref{over} admits a solution,  $\Omega$ is a ball centered at the origin.
\end{theorem}

 To prove Theorem \ref{overth}, we will need the following technical lemma.

\begin{lemma}[Monotonicity] \label{mono}
Let $\Omega_1\subset \Omega_2$ be two bounded and open domains in $\mathbb R^n$, $n\geq 1$, and let $u_i$ be the continuous weak solution to
\begin{align*}
\begin{cases}
(-\Delta_{p})^s u_1=1 & \text{ in } \Omega_1\\
u_1=0 & \text{ in } \mathbb R^n\setminus \Omega_1
\end{cases},
\qquad
\begin{cases}
(-\Delta_{p})^s u_2=1 & \text{ in } \Omega_2\\
u_2=0 & \text{ in } \mathbb R^n\setminus \Omega_2.
\end{cases}
\end{align*}Then $u_1\leq u_2$ in $\mathbb R^n$.
\end{lemma}
\begin{proof}
By the strong maximum principle \cite[Lemma 12]{LL14} we can assume that $u_1, u_2\geq 0$ and since $u_1=0$ in $\mathbb R^n\setminus \Omega_1$ we have that $u_2\geq 0=u_1$ in $\mathbb R^n\setminus \Omega_1$. 

Moreover, given any nonnegative continuous function $\psi \in W^{s,p}_0(\Omega_1)$
$$
\langle (-\Delta_{p})^s  u_2 ,\psi \rangle = \int_{\Omega_2} \psi\,dx \geq \int_{\Omega_1}\psi\,dx, \qquad \langle (-\Delta_{p})^s  u_1,\psi \rangle = \int_{\Omega_1} \psi\,dx
$$
then
$$
\langle (-\Delta_{p})^s u_2,\psi \rangle - \langle (-\Delta_{p})^s u_1,\psi \rangle \geq 0
$$
and the comparison principle \cite[Lemma 9]{LL14} gives that $u_2\geq u_1$ in $\mathbb R^n$.
\end{proof}

\begin{proof}[Proof of Theorem \ref{overth}]
Let $u$ be a solution of \eqref{over}. Let us see that $\Omega$ is a ball. Define the radii $R_1\leq R_2$
$$
R_1=\min_{z\in\partial \Omega} |z|, \qquad R_2=\max_{z\in\partial \Omega} |z|. 
$$
Then, $B_{R_1}$ is the largest ball centered at the origin and contained in $\Omega$, and $B_{R_2}$ is the smallest ball centered at the origin and containing $\Omega$, and there exist $z_i\in\partial\Omega$, $i=1,2$ satisfying  $|z_i|=R_i$.

Let us see that $R_1=R_2$ and therefore we may conclude that $\Omega$ is a ball. Denote by $u_{R_i}$ the solution of \eqref{torsionR} when $R=R_i$, $i=1,2$. By Lemma \ref{mono}, these solutions as ordered as follows
\begin{equation} \label{equ1}
 u_{R_1}\leq u \leq u_{R_2} \quad \text{ in } \mathbb R^n.  
\end{equation}
Let $\nu_1$ the outer normal to $\partial B_{R_1}$ at $z_1$. Observe that the point $x=z_1-t\nu_1$, $t\in [0,R_1]$ runs along the ray $r$ of $B_{1}$ passing through $z_1$, then
\begin{equation} \label{equ2}
\delta_{R_1}(x)=|x-z_1|=\delta_\Omega(x).
\end{equation}
From \eqref{equ1} and \eqref{equ2} we get
$$
\rho_s(R_1)=\lim_{B_{R_1}\ni x\to z_1} \frac{u_{R_1}(x)}{\delta_{R_1}(x)^s} \leq 
\lim_{\Omega\ni x\to z_1} \frac{u(x)}{\delta_{\Omega}(x)^s} =q(|z_1|)=q(R_1).
$$
Similarly, since $\Omega$ satisfies the interior ball condition, there exists $B_R\subseteq \Omega$ with $z_2\in\partial B_R\subseteq B_{R_2}$. Then the outer normal $\nu_2$ to $\partial B_{R_2}$ at $z_2$ is also normal to $\partial B_R$. Letting $x=z_2-t\nu_2$ with $t\in [0,R]$ we have that
\begin{equation} \label{equ3}
\delta_{R_2}(x)=|x-z_1|=\delta_\Omega(x),
\end{equation}
and
$$
q(R_2)=\lim_{B_{R_2}\ni x\to z_2} \frac{u(x)}{\delta_{\Omega}(x)^s} \leq 
\lim_{\Omega\ni x\to z_2} \frac{ u_{R_2}(x)}{\delta_{R_2}(x)^s} = \rho_s(R_2).
$$
Hence, this analysis leads to
$$
\frac{q(R_2)}{ \rho_s(R_2)} \leq 1 \leq \frac{q(R_1)}{ \rho_s(R_1)}.
$$
Since the ratio $\frac{q(r)}{\rho_s(r)}$ is strictly increasing for $r>0$, $R_1=R_2$ and then $\Omega=B_{R_1}=B_{R_2}$, that is, $\Omega$ is a ball.
\end{proof}

%\subsection*{Acknowledgements.} 

\end{document}